\documentclass[a4paper,12pt, leqno]{amsart}
\usepackage{amsmath,amssymb,amsthm,amsfonts,amstext,amscd}
\usepackage[margin=2.5cm]{geometry}
\usepackage[T1]{fontenc}
\usepackage{newlfont}
\usepackage{longtable}
\usepackage[all]{xy}
\usepackage{hyperref}
\usepackage[ansinew]{inputenc}
\usepackage{color}
\usepackage[english]{babel}

\newcommand{\nc}{\newcommand}

\theoremstyle{plain}
\newtheorem{theorem}{\sc Theorem}[section]
\newtheorem{thm}[theorem]{\sc Theorem}
\newtheorem{lem}[theorem]{\sc Lemma}
\newtheorem{prop}[theorem]{\sc Proposition}
\newtheorem{cor}[theorem]{\sc Corollary}

\newtheorem{rem}[theorem]{\sc Remark}
\newtheorem{ex}[theorem]{\sc Example}

\newtheorem{alg}[theorem]{\sc Algorithm}

\theoremstyle{definition}
\newtheorem{define}[theorem]{\sc Definition}

\renewcommand{\refname}{\center{REFERENCES}}

\newcommand{\vfi}{\varphi}

\newcommand{\N}{\mathbb{N}}
\newcommand{\Z}{\mathbb{Z}}

\nc{\GG}{\mathfrak{G}}

\nc{\LL}{\mathcal{L}}
\nc{\JJ}{\mathcal{J}}

\nc{\NN}{\mathfrak{N}}
\nc{\HH}{\mathfrak{H}}
\nc{\cH}{\mathcal{H}}
\nc{\cG}{\mathcal{G}}
\nc{\cN}

\nc{\pgs}{\mathfrak{p}\mathfrak{g}\mathfrak{s}}
\nc{\pcr}{\mathfrak{p}\mathfrak{c}\mathfrak{r}}
\nc{\wg}{\widehat{g}}
\nc{\wh}{\widehat{h}}
\nc{\wx}{\widehat{x}}
\nc{\wy}{\widehat{y}}

\nc{\Char}{\operatorname{char}}
\nc{\Ker}{\operatorname{Ker}}
\nc{\Imm}{\operatorname{Im}}
\nc{\Aut}{\operatorname{Aut}}
\nc{\Cl}{\operatorname{Cl}}
\nc{\Orb}{\operatorname{Orb}}
\nc{\noreq}{\trianglelefteq}

\advance\textheight by \topskip
\title[$q$-Tensor Square of a Polycyclic Group]
{A Polycyclic Presentation for the $q$-Tensor Square of a Polycyclic Group}
\author[Dias]{Ivonildes Ribeiro Martins Dias}
\address{Instituto de Matem\'atica e Estat\'istica,
		 Universidade Federal de Goi\'as, Goi\^ania-GO, 74001-970 Brazil}
\email{ivonildes@ufg.br}
\author[Rocco]{Nora\'i Romeu Rocco*}
\address{Departamento de Matem\'atica-IE, Universidade de Bras\'ilia,
Bras\'ilia-DF, 70910-900 Brazil}
\email{norai@unb.br}
\thanks{(*) The author acknowledges partial financial support from FAPDF, Brazil, during the preparation of this
work.}
\date{\today}

\begin{document}

\begin{abstract}
Let $G$ be a group and $q$ a non-negative integer. We denote by $\nu^q(G)$ a certain extension of the $q$-tensor square
$G \otimes^q G$ by $G \times G$. In this paper we derive a polycyclic presentation for $G \otimes^q G$,  when $G$ is
polycyclic, via its embedding into $\nu^q(G)$. Furthermore, we derive presentations for the
$q$-exterior square $G \wedge^q G$ and for the second homology group $H_2(G, \Z_q).$
Additionally, we establish a criterion for computing the $q-$exterior centre $Z_q^\wedge (G)$ of a polycyclic group $G,
$ which is helpful for deciding whether $G$ is capable modulo $q$. These
results extend to all $q \geq 0$ existing methods due to Eick and Nickel for the case $q = 0$.
\end{abstract}

\maketitle

\section{Introduction} \label{sec:1}

Let $G$ be a group and $q$ a non-negative integer. The $q$-tensor square $G \otimes^q G$ is a particular case of the
$q$-tensor product $G \otimes^q H$ of groups $G$ and $H$ which act compatibly on each other; this construction was
defined by Conduch\'e and Rodrigues-Fernandez in \cite{Conduche}, in the context of $q$-crossed modules (see also
\cite{Ellis}, \cite{McDermott} and \cite{Brown}). It reduces to Brown and Loday's non-abelian tensor product
$G \otimes H$ when $q = 0$ (cf. \cite{BL}).

For $x, y \in G$, we write the conjugate of $y$ by $x$ as $y^x = x^{-1} y x$; the commutator of $x$ and $y$ is then
written as $[x, y] = x^{-1}y^{-1}xy$. Commutators are left normed: $[x, y, z] = [[x,y],z]$, and so on for commutators
of higher weights.

For $q \geq 1$ let $\widehat{\cG} := \{\widehat{k} \mid k \in G\}$ be a set of symbols, one for each element of $G$.
According to Ellis \cite{Ellis}, the $q$-tensor square $G \otimes^q G$ is then defined to be the group
generated by all symbols
$g \otimes h$ and $\widehat{k}$, $g, h, k \in G$, subject to the following defining relations:
\begin{eqnarray} \label{qts1}
(g\otimes h h_{1}) = (g \otimes h )(g^{h_{1}}\otimes h^{h_{1}})
\end{eqnarray}
\begin{eqnarray} \label{qts2}
gg_{1}\otimes h = (g^{g_{1}}\otimes h^{g_{1}})  (g_{1}\otimes h )
\end{eqnarray}
\begin{eqnarray} \label{qts3}
(g \otimes h )^{\widehat{k}} = (g^{k^{q}}\otimes h^{k^{q}})
\end{eqnarray}
\begin{eqnarray} \label{qts4}
{\widehat{kk_{1}}}= \widehat{k} \displaystyle \prod _{i=1}^{q-1} \left(k
\otimes  (k_{1}^{-i})\right)^{k^{q-1-i}}{\widehat k_{1}}
\end{eqnarray}
\begin{eqnarray} \label{qts5}
\left[ \widehat{k}, \widehat{k_{1}}\right]= k^{q} \otimes {k_{1}}^{q}
\end{eqnarray}
\begin{eqnarray} \label{qts6}
{\widehat{\left[g,h\right]}}= (g\otimes h )^{q}
\end{eqnarray}
for all $g, g_1$,$h, h_1$, $k, k_1 \in G$. If $q=0$ then we set $\widehat{\cG} = \emptyset$ to get the group generated
by the symbols $g \otimes h, \; g, h \in G,$ subject to the relations \eqref{qts1} and \eqref{qts2} only; that is,
$G \otimes^0 G$ is the non-abelian tensor square $G \otimes G$.
By the defining relations \eqref{qts1} -- \eqref{qts6} we see that the diagonal $\bigtriangledown^q(G) = \langle g
\otimes g \mid g \in G \rangle$ is a central subgroup of $G \otimes^q G$. The $q$-exterior square $G \wedge^q G$ is by
definition the factor group (see \cite{ER89}):
\begin{equation} \label{def:qext}
	G \wedge^q G = \dfrac{G \otimes^q G}{\nabla^q(G)}.
\end{equation}
We usually write $g \wedge h$ for the image of $g \otimes h$ in $G \wedge^q G$.

There is a map
\begin{equation} \label{eq:varrho}
\varrho :  G \otimes^q G \to G, \;  g \otimes h \mapsto [g, h], \;  \widehat{k} \mapsto k^q,
\end{equation}
for all $g, h, k  \in G$. Clearly, $\nabla^q(G) \leq \Ker\varrho$ and we have (see for instance
\cite[Proposition 18]{Brown} or \cite[Theorem 2.12]{BR}):
\begin{equation} \label{eq:H2}
\Ker\varrho /\nabla^q (G) \cong H_2(G, \Z_q),
\end{equation}
the second homology group of $G$ with coefficients in the trivial $G$-module $\Z_q$.
The image $\Imm \varrho$ is the subgroup $G^\prime G^q \leq G$, where $G^\prime$ is the derived subgroup of $G,$
generated by all commutators $[g, h]$ with $g,h \in G$,  and $G^q$ is the subgroup of $G$ generated by all $q-$th
powers
$g^q, \, g \in G$. Thus, we get the exact sequence (cf. \cite[Proposition 18]{Brown}):
\begin{equation} \label{eq:brown}
1 \to H_2(G, \Z_q) \to G \wedge^q G \to G^\prime G^q \to 1.
\end{equation}
A group $G$ is called $q$-{\em perfect} in case $G = G^\prime G^q$. If this is the case, then
 the above sequence shows that $G \wedge^q G$ is a $q$-central extension of $G$ and
in addition it is the unique universal $q-$central extension of $G$ (see \cite{Brown}). Notice that
if $G$ is $q-$perfect then $G \otimes^q G \cong G \wedge^q G$.

The $q$-exterior square is also helpful in deciding whether a group $G$ is $q$-capable; recall that
$G$ is $q$-capable if there exists a group $Q$
such that $Z(Q) = Z_q(Q)$ and $G \cong Q/Z(Q)$, where $Z(Q)$ is the center of $Q$ and $Z_q(Q)$ is the
$q$-center, that is, the elements of the center $Z(Q)$ of order dividing $q$. The
 $q$-exterior center of $G$ is the subgroup of $G$ defined by
\begin{equation} \label{eq:extcenter}
	Z_q^{\wedge}(G) = \left\{g \in G \; \vert \; g \wedge x = 1 \in G \wedge^q G,  \; \text{for all}\;  x \in G \right\}.
\end{equation}
In \cite[Proposition 16]{Ellis} Ellis proved that the group $G$ is $q$-capable if, and only if, $Z_q^\wedge(G)=1$.

So getting a presentation for the $q$-tensor square of a group $G$ and for its subfactors is an interesting task.

It's known that if $G$ is a polycyclic group then $G \otimes^q G$ is polycyclic, for all $q \geq 0$
(see for instance \cite{BR}). In \cite{EN} the authors describe algorithms to compute the non-abelian tensor square
$G \otimes G$,
the exterior square $G \wedge G$ and the Schur multiplier
$M(G),$ among others, for a polycyclic group $G$ given by a consistent polycyclic presentation; the implementation of
this algorithm is available in \cite{ENGAP}. They manage to find such an algorithm to computing $G \otimes G$   by
finding a presentation of the group $\nu(G)$, as introduced for instance in \cite{Rocco} (see also \cite{EL}), which
turns out to be an extension of
$G \otimes G$ by $G \times G$.

The present paper aims to extend that algorithms to all $q \geq 0$. Instead of group $\nu(G)$ we now consider the group
$\nu^q(G)$, defined for instance in \cite[Definition 2.1]{BR}. To ease reference we briefly describe this group early
in the next section.

The paper is organized as follows.

In Sec. 2 we describe some basic constructs and preliminaries results. In Sec. 3 we give consistent polycyclic
presentations for certain q-central extension of $G$, more specifically, the groups $E_q(G)$ and $\mathfrak{E}^q(G).$
In Sec. 4 we give polycyclic presentations for the second homology group $H_2(G, \mathbb{Z}_q)$ and for the $q$-
exterior square $G \wedge^q G$; use of the group $E_q(G)$ is made to exemplify the computation of the
q-exterior centre of $G$.
In Sec. 5 we provide a consistent polycyclic presentation for the group $\tau^q(G)$. Finally, in Sec. 6 we give an
algorithm to compute a polycyclic presentation for $\nu^q(G)$ and the q-tensor square of a polycyclic group $G$.

Notation is fairly standard; for basic results on Group Theory, see for instance \cite{Robinson}. In this article all
group actions are on the right. Basic notation and
structural results concerning  $\nu^q(G)$ can be found for instance in \cite{BR}.

The main content of this article is part of the doctoral thesis~\cite{Ribeiro} of the first named author,
elaborated under the supervision of the second.

\section{Preliminary Results} \label{sec:2}

We begin this section by defining the group $\nu^q(G)$ and giving a brief description of some of its properties.

To this end, let $G^\vfi$ be an isomorphic copy of $G,$ via an isomorphism $\varphi$ such
that $\vfi: g \mapsto g^\vfi$, for all $g \in G.$ With these data we immediately get the group $\nu(G)$, as mentioned
before, defined as follows:
\begin{equation} \label{eq:presentanu}
 \nu(G) : = \left \langle G \cup G^{\varphi} \, | \,
[g, h^{\varphi}]^k = [g^k, (h^k)^{\varphi}] = [g, h^{\varphi}]^{k^{\varphi}}, \,
 \forall g,h,k \in G \right \rangle.
\end{equation}
It's well known (see \cite{Rocco}, and also \cite{EL}) that the subgroup $[G, G^\varphi]$ of $\nu(G)$ is isomorphic to
the non-abelian tensor square $G \otimes G$, so that the strategy of finding an appropriate representation of $\nu(G)$
can be useful to compute $G \otimes G$ and various of its relevant subfactors (see for instance \cite{EL},
\cite{Rocco}, \cite{McDermott}, \cite{BM} and \cite{EN}).

Now for $q \geq 1$, let $\widehat{\mathcal{G}} = \{\widehat{k} \, | \, k \in G \}$ be a set of symbols, one for each
element of $G$ (for $q = 0$ we set $\widehat{\mathcal{G}} = \emptyset,$ the empty set) and
let $F(\widehat{\mathcal{G}})$ be the free group over $\widehat{\mathcal{G}}.$
Write $\nu(G) \ast F(\widehat{\mathcal{G}})$ for
the free product of $\nu(G)$ and $F(\widehat{\mathcal{G}}).$ As $G$ and $G^{\varphi}$
are embedded into $\nu(G)$ we shall identify the elements of $G$ (respectively of $G^{\varphi}$)
with their respective images in $\nu(G) \ast F(\widehat{\mathcal{G}})$.
Denote by $J$ the normal closure in $\nu(G) \ast F(\widehat{\mathcal{G}})$
of the following elements, for all $\widehat{k}, \widehat{k_1} \in
\widehat{\mathcal{G}}$ and $\ g, \, h  \in G:$
 \begin{gather}
{g}^{-1} \, \widehat{k} \, g \; \widehat{(k^g)}^{-1}; \label{RR1} \\
(g^{\varphi})^{-1} \, \widehat{k} \, g^{\varphi} \; \widehat{(k^g)}^{-1};
\label{RR2} \\
(\widehat{k})^{-1} [g ,h^{\varphi}] \, \widehat{k} \; [g^{k^q}, (h^{k^{q}})^{\varphi}]^{-1};
\label{RR3} \\
(\widehat{k})^{-1} \, \widehat{k k_1} \; (\widehat{k_1})^{-1}
 \displaystyle({\prod_{i=1}^{q-1}}[k, (k^{-i}_1)^{\varphi}]^{k^{q-1-i}})^{-1};  \label{RR4} \\
 [\widehat{k}, \widehat{k_1}] \; [k^q, (k^{q}_1)^{\varphi}]^{-1}; \label{RR5} \\
 \widehat{[g, h]} \; [g, h^{\varphi}]^{-q}.  \label{RR6}
 \end{gather}

\begin{define} \label{nu^q} \label{def:nuq}
The group $\nu^q(G)$ is defined to be the factor group
\begin{equation}
\nu^q(G) := ( \nu(G) \ast F(\widehat{\mathcal{G}}) )/J.
\end{equation}
\end{define}
Note that for $q=0$ the sets of relations $(\ref{RR1})$ to
$(\ref{RR6})$ are empty; in this case we have $\nu^{0}(G) = \nu(G) \ast
F(\widehat{\mathcal{G}}))/J \cong \nu(G)$.
Let $R_1, \ldots, R_6$ be the sets of relations corresponding to
$(\ref{RR1}), \ldots, (\ref{RR6})$, respectively, and let $R$ be their union,
$R= \bigcup^6_{i=1}R_i$. Therefore, $\nu^q(G)$ has the presentation:
 \begin{equation*}
\nu^q(G)= \left \langle G, G^{\varphi}, \widehat{\mathcal{G}} \; | \; R,
[g,h^{\varphi}]^k \, [g^k, (h^k)^{\varphi}]^{-1},
[g, h^{\varphi}]^{k^{\varphi}} \, [g^k, (h^k)^{\varphi}]^{-1}, \, \forall g, h,
k \in G \right \rangle.
 \end{equation*}
The above presentation of $\nu^q(G)$ is a variant of the one given by Ellis in
\cite{Ellis}.

There is an epimorphism $\rho: \nu^q(G) \twoheadrightarrow G,  g \mapsto  g, h^{\varphi} \mapsto h,
 \widehat{k} \mapsto k^q$. On the other hand the inclusion of
 $G$ into $\nu(G)$ induces a homo\-mor\-phism $\imath: G \to \nu^q(G)$. We have
 $g^{\imath \rho} = g$ and thus $\imath$ is injective.
 Similarly the inclusion of $G^{\vfi}$ into $\nu(G)$ induces a monomorphism
$\jmath: G^{\vfi} \to \nu^q(G)$. These embeddings allow us to identify the elements $g \in G$
and $h^{\vfi} \in G^{\vfi}$ with their respective images $g^{\imath}$ and
$(h^{\vfi})^{\jmath}$ in $\nu^{q}(G)$.

Now let $\GG$ denote the subgroup of $\nu^q(G)$ generated by the
images of $\widehat{\cG}$. By relations $(\ref{RR3})$,  $\GG$
normalizes the subgroup $[G, G^{\varphi}]$ in $\nu^q(G)$ and hence
${\Upsilon}^{q}(G) := [G, G^{\varphi}] \GG$ is a normal subgroup of
${\nu}^q(G)$. Hence we get $\nu^q(G) = G^{\varphi} \cdot (G \cdot \Upsilon^q(G)),$ where the dots
mean internal semidirect products.

By \cite[Proposition 2.9]{BR} there is an isomorphism
$\mu: {\Upsilon}^{q}(G) \to G \otimes^q G$ such that
$[g, h^\varphi] \mapsto g \otimes h, \; \widehat{k} \mapsto \widehat{k}$,
for all $g, h, k \in G$ and for all $q \geq 0$. We then get (see \cite[Corollary 2.11]{BR})
\begin{equation} \label{eq:decompose_nu}
\nu^q(G) \cong G \ltimes (G \ltimes (G \otimes^q G));
\end{equation}
This decomposition of $\nu^q(G)$ is analogous to one due to Ellis in \cite{Ellis}; it generalizes a
similar result found in \cite{Rocco} for $q=0.$

In view of the above isomorphism, unless otherwise stated
from now on we will identify $G \otimes^q G$ with the subgroup
$\Upsilon^q(G) = [G, G^\vfi] \GG \leq \nu^q(G)$ and write
$[g, h^\vfi]$ in place of $g \otimes h$, for all
$g, h \in G$.
Following \cite{BR} we write $\Delta^q(G)$ for the subgroup
$\langle [g, g^\vfi] \vert g \in G \rangle \leq \Upsilon^q(G),$
which by Remark~\ref{rem:action} below is a central subgroup of $\nu^q(G).$
The isomorphism $\mu$ restricts to an isomorphism $\Delta^q(G) \stackrel{\mu}{\cong} \nabla^q(G)$
and, consequently, the factor group $\Upsilon^q(G)/ \Delta^q(G)$
is isomorphic to the q-exterior square $G \wedge^q G$.
In this case, as usual we simply write $g \wedge h$ to denote the coset
$[g, h^\varphi] \Delta^q(G)$ in  $G \wedge^q G.$
We shall eventually write $\mathcal{T}$ to denote the subgroup
$[G, G^\vfi]$ of ${\nu^q(G)}$ in order to distinguish it from the non-abelian tensor square
$G \otimes G \equiv [G, G^\vfi] \leq \nu(G)$ in the case $q =0$. We also write $\tau^q(G)$
for the factor group $\nu^q(G)/ \Delta^q(G);$ thus we get
\begin{equation*} \label{eq:tauq}
\tau^q(G) \cong ((G \wedge^q G) \rtimes G) \rtimes G.
\end{equation*}

\begin{rem} \label{rem:action}
It should be noted that the actions of $G$ and
$G^{\vfi}$ on $\Upsilon^q(G)$ are those induced by the defining relations of
$\nu^q(G)$: for any elements $g, x \in G$, $h^{\vfi}, y^{\vfi} \in G^{\vfi}$
and $\widehat{k} \in \widehat{\mathcal{G}}$, we have
$[g, h^{\vfi}]^{x} = [g^{x}, (h^{x})^{\vfi}]$ and
$(\widehat{k})^{x} = \widehat{(k^x)}$. In view of the isomorphism $\Upsilon^q(G) \equiv G \otimes^q G$, these
correspond to the action of $G$ on $G \otimes^q G$ as given for instance in \cite{Ellis}:
$$\left\{
	\begin{array}{cccc}
	(g \otimes h)^x = g^x \wedge h^x \\
	\; (\widehat{k})^x = \widehat{k^x} .\\
	\end{array}\right.
$$
Similarly,
$[g, h^{\vfi}]^{y^{\vfi}} = [g^{y}, (h^{y})^{\vfi}]$ and
$(\widehat{k})^{y^{\vfi}} = \widehat{(k^y)}.$ In addition, for any $\tau \in
\Upsilon^q(G)$, $(g \tau)^{y^{\vfi}} = g [g, y^{\varphi}] \tau^{y^{\vfi}} \in G \Upsilon^q(G)$.
Similar actions are naturally induced on the q-exterior square $G \wedge^q G$.
\end{rem}

It is known that if $G$ is polycyclic, then $\nu^q(G)$ is polycyclic for all $q \geq 0$ and thus,
as mentioned before, $G \otimes^q G$ is polycyclic.
In \cite{BR} the authors proved that, for a polycyclic group $G$ given by a consistent polycyclic presentation, the
defining relations of $\nu^q(G)$ can be reduced to relations among the polycyclic generators, with the only exception
of relations~\eqref{RR4} which have a more complicated handling characteristic. Even so, they were able to use the GAP
System~\cite{gap} to compute $\nu^q(G),$ $G \otimes^q G$ and
$G \wedge^q G$, for some small groups $G$ and particular values of $q.$  In addition, in \cite{RR} it is given a
description of the q-tensor square of a n-generator nilpotent group of class 2, $n > 1$, for all $q > 1$ and $q$ odd.

Our purpose in this article is to overcome in some way the difficulty of dealing with relations~\eqref{RR4} and give a
polycyclic presentation for the groups $\nu^q(G)$, $G \otimes^q G$, $G \wedge^q G$ and $H_2(G, \Z_q)$, for all
$q \geq 0$, when $G$ is polycyclic given by a consistent polycyclic presentation. Our approach is based on ideas of
Eick \& Nickel~\cite{EN} for the case $q = 0$.

The concept of a crossed pairing (biderivation) has been used in order to determine homomorphic images of the
non-abelian tensor square $G \otimes G$ (see~\cite[Remark 3]{BJR}). We need to extend this concept in order to the
context of the q-tensor square.

\begin{define} \label{def:qbide}
Let $G$ and $L$ be arbitrary groups and $q$ a non-negative integer.
A function $\lambda: G \times G \times G \to L$ is called a $q$-{\em biderivation} if the following properties hold:
	\begin{eqnarray}\label{relqbide}
	(gg_1, h, k)\lambda = (g^{g_1},h^{g_1},1)\lambda \, (g_1,h,k)\lambda
	\end{eqnarray}
	\begin{eqnarray} \label{relqbide2}
	(g,hh_1,k)\lambda = (g,h_1,1)\lambda \,  (g^{h_1},h^{h_1},k)\lambda
	\end{eqnarray}
	\begin{eqnarray} \label{relqbide3}
	\left((1,1,k)\lambda\right)^{-1}(g,h,1)\lambda \, (1,1,k)\lambda=(g^{k^q},h^{k^q},1)\lambda
	\end{eqnarray}
	\begin{eqnarray} \label{relqbide4}
	(1,1,kk_1)\lambda = (1,1,k)\lambda \, \prod_{i=1}^{q-1}\left\{(k,(k_1^{-i})^{k^{q-1-i}},1)\lambda \right\} \,
	(1,1,k_1)\lambda
	\end{eqnarray}
	\begin{eqnarray} \label{relqbide5}
	\left[(1,1,k)\lambda,(1,1,k_1)\lambda\right]=(k^q,k_1^q,1)\lambda
	\end{eqnarray}
	\begin{eqnarray} \label{relqbide6}
	(1,1,[g,h])\lambda = \left((g,h,1)\lambda\right)^q
	\end{eqnarray}
	for all $g,g_1,h, h_1,k_1 \in G$.
\end{define}

By the defining relations \ref{def:nuq} of $\nu^q(G)$, it is easy to see that a $q$-biderivation provides a
universal property of the $q$-tensor square of a group $G$; we record this property as

\begin{prop} \label{qbide}
Let $G$ and $L$ be arbitrary groups and $\lambda : G \times G \times G \to L$ a $q$-biderivation. Then, there
exists a unique homomorphism $\widetilde{\lambda}: G \otimes^q G \to L$ such that the following hold, for all
$g, h, k \in G$:
\[
(g \otimes h) \widetilde{\lambda} = (g,h,1) \lambda;
\]
\[
(\widehat{k}) \widetilde{\lambda} = (1,1,k) \lambda.
\]
\end{prop}

To ease reference we include the next Lemma, which relates the $q$-exterior square of $G$ and the second homology group
$H_2(G, \Z_q)$ with an arbitrary free presentation $F/R$ of $G$ (see \text{\cite{McDermott}} and also \cite{ER89}).

\begin{lem}  \label{lem:liv}
Let $F/R$ be a free presentation for the group $G$. Then,
\[
G \wedge^q G \cong \frac{F^\prime F^q}{[ R, F] R^q} \text{ and } H_2(G, \Z_q) \cong \dfrac{R \cap F'F^q}{ [R, F] R^q}.
\]
\end{lem}

Thus we have $H_2(G, \Z_q) \cong (G \wedge^q G) \cap M^q(G),$ where $M^q(G) = R/ [R, F] R^q$ is the
$q$-{\em multiplier}.

Notice that when $F$ is a free group, then we find that $F \wedge^q F \cong F^\prime F^q.$ A similar result is
also valid for projective q-crossed $G$-modules:

\begin{prop} \cite[Proposition 1.3.11]{McDermott} \label{qiso} Let $\delta : M \rightarrow G$ be a projective $q$-
crossed $G$-module
and let $F/R$ be a free presentation of $G$, with $\pi : F \rightarrow G$ being the natural epimorphism. Then there
exists an isomorphism
\[
M' M^q \stackrel{\cong}{\rightarrow} F^\prime F^q/ [R, F] R^q,
\]
such that $[m, m']^\delta = [f,f'] [R, F] R^q$ and $(m^q)^\delta=f^q [F,R] R^q$, where $(m)^\delta = (f)^\pi$.
\end{prop}

\section{Consistent polycyclic presentations for the groups
$E_q(G)$ and $\mathfrak{E}^q(G)$,
$q$-central extensions of $G$}
\label{gast}

 In this section we describe a method for computing consistent polycyclic presentations for certain $q$-central
extensions of a polycyclic group $G$ given by a consistent polycyclic presentation. Our method is a generalization of
the one given by Eick and Nickel in \cite{EN} for the case $q = 0$.

Let $G$ be a polycyclic group defined by a consistent polycyclic presentation $F_n/R$, where $F_n$ is the free group
in the generators $g_1,...,g_n$ and let $H$ be a finitely presented group defined by a finite presentation $F_m/S$,
where $F_m$ is the free group on the generators $f_1,...,f_m$. For our purposes we shall assume that $m \leq n$.
Suppose that  $\xi : H \rightarrow G$ is an epimorphism, such that $(f_i)\xi = w_i, \; 1 \leqslant i \leqslant m,$
where $w_i$ is a word in the generators $g_1, \ldots, g_n$. Denote by $K/S$ the kernel $\Ker \xi.$
Thus, $G \cong F_m/K$.
Define the groups
 \[
E_q(G):= \dfrac{F_n}{R^q[F_n,R]}
\]
and
\[
\mathfrak{E}^q(G):=\dfrac{F_m}{K^q[K, F_m]S}
\]
which, by construction, are $q$-central extensions of $G$.

The following result in the context of crossed modules will be helpful.

\begin{prop} \cite[Lemma 5.2.2]{McDermott} \label{resu} With the above definition, the natural epimorphism
$\pi : E_q(G) \rightarrow G$ is a projective $q$-crossed $G$-module.
 \end{prop}

The relations of a consistent polycyclic presentation $F_n/R$ have the form:

 \begin{itemize}
 	\item[]\hspace{4cm} \hspace{0.95cm}$ g_i^{e_i}= g_{i+1}^{\alpha_{i,i+1}}... g_n^{\alpha_{i,n}}$ for $i \in I$,
 	\item[] \hspace{4cm}$g_j^{-1}g_ig_j=g_{j+1}^{\beta_{i,j,j+1}}...g_n^{\beta_{i,j,n}}$ para $j<i$,
 	\item[]\hspace{4cm} $g_jg_ig_j^{-1}=g_{j+1}^{\gamma_{i,j,j+1}}...g_n^{\gamma_{i,j,n}}$ para $j<i$ e
 	$j \notin I$,
 \end{itemize}
 for some set $I \subseteq \{1,...,n\}$, certain exponents $e_i \in \N$ for $i \in I$, and $\alpha_{i,j}$,
$\beta_{i,j,k}$, $\gamma_{i,j,k} \in \Z,$ for all $i$, $j$ and $k$. To ease notation we shall write the defining
relations of $G$ as relators, in the form $r_1,....,r_l$. Thus, each relator $r_j$ is a word in the generators
$g_1,...,g_n$; that is, $r_j = r_j(g_1,...,g_n)$.

 We now introduce $l$ new generators $t_1,...,t_l,$ one for each relator $r_j$, and define a new group $\epsilon(G)$ to
be the group generated by
$g_1,...,g_n, t_1,...,t_l$, subject to the relators:
 \begin{itemize}
 	\item[(a)] $r_i(g_1,...,g_n)t_i^{-1},$ for $1\leq i\leq l$,
 	\item[(b)] $[t_i,g_j],$ for $1\leq j \leq n$, $1\leq i \leq l$,
 	\item[(c)] $[t_i,t_j],$ for $1\leq j < i \leq l$,
 	\item[(d)] $t_i^q,$ for $1\leq i\leq l$.
 \end{itemize}

 Denote by  $T_q$ the $q$-central subgroup of $\epsilon(G)$ generated by $\left\{t_1,...,t_l\right\}.$ It follows
directly from these relators that $\epsilon(G)$ is a $q$-central extension of $G$ by \index{$T_q$} $T_q$.

The following Lemma asserts that the above relations give a polycyclic presentation of
$E_q(G)$, possibly inconsistent.

 \begin{lem}\label{lem1} Let $G$ be a polycyclic group given by a consistent polycyclic presentation $F_n/R$.
Then we have:
 	$\epsilon(G) \cong F_n/R^q[R,F_n]$, $T_q \cong R/R^q[R,F_n]$ and $\epsilon(G)/T_q \cong G$.
 \end{lem}
 \begin{proof} It follows by relations $(a)$ above that $\epsilon(G)/T_q \cong  G$, while by relations $(b)$, $(c)$ and
$(d)$ we immediately see that $T_q$ is a $q$-central subgroup of $\epsilon(G).$
Define $\sigma: F_n\rightarrow \epsilon(G)$ given by $(g_i)^\sigma = g_i$, $1 \leqslant i \leqslant n$.
Relations $(a)$
imply that $\sigma$ is an epimorphism and, since $\epsilon(G)$ is a $q$-central extension of $G,$ we have that
$R^q[R,F_n] \leqslant Ker(\sigma) \leqslant R$.
 	On the other side, there exists a well defined homomorphism from $\epsilon(G)$ to $\frac{F_n}{R^q[F_n,R]}$,
which is an epimorphism. Consequently, $Ker(\sigma) \leq R^q[R,F_n] \leq Ker(\sigma)$ and thus,
$\epsilon(G) \cong \frac{F_n}{R^q[F_n,R]}$. Therefore,
$\epsilon(G) \cong E_q(G)$, where we get $T_q\cong \frac{R}{R^q[R,F_n]}$.
 \end{proof}

  By using an adaptation of the method described by Eick and Nickel in \cite{EN} (see also \cite[p. 424]{Sims}) we can
determine a consistent polycyclic presentation for $E_q(G)$ from the
(possibly inconsistent) polycyclic presentation
 given by Lemma~\ref{lem1}. We then get a consistent polycyclic presentation for $E_q(G)$
in the generators $g_1,...,g_n, t_1,...,t_l$ with the following relations:
 \begin{itemize}
 	\item[(1)] $r_i(g_1,...,g_n)t_1^{q_{i1}}...t_l^{q_{il}}$, for $1\leq i\leq l$;
 	\item[(2)] $[t_i,g_j]$, for $1 \leq i \leq n$, $1\leq j \leq l$;
 	\item[(3)] $[t_i, t_j],$ for $1\leq j < i \leq l$;
    \item[(4)] $t_i^{d_i},$ for $1\leq i\leq l,$ with $d_i \mid q,$
 \end{itemize}
where $(q_{ij})_{1 \leq i, j \leq l}$ is an appropriate invertible matrix over $\mathbb{Z}$.
 It may happen that $d_i=1$ for some $i \in \{1,...,l\}$. In this case the corresponding generator
$t_i$ is redundant and can be removed.

 Bellow we give a couple of simple examples in order to illustrate these results. The same examples will be
used in subsequent sections.
 \begin{ex} \label{ex3.1}
 First we consider the symmetric group $S_3$, given by the consistent polycyclic presentation
$$S_3 = \left\langle g_1, g_2 \mid g_1^2=1,  g_1^{-1}g_2g_1=g_2^2, g_2^3=1\right\rangle.$$
According to the definition we have, say for $q = 2$:
    \[
 	E_2(S_3)  = \left\langle g_1, g_2, t_1, t_2, t_3  \mid g_1^2 = t_1,
 	g_1^{-1}g_2g_1 = g_2^2t_2,  g_2^3 = t_3,
    t_1^2 = 1,   t_2^2 = 1,  t_3^2 = 1
    \right\rangle,
    \]
   where $t_1, t_2, t_3$ are  central.
  Checking for consistency we find that $t_2=1$. Thus, a consistent polycyclic presentation of $E_2(S_3)$ is
\[
E_2(S_3)= \left\langle g_1,g_2, t_1, t_3 {\cal j}g_1^2=t_1,
g_1^{-1}g_2g_1=g_2^2,  g_2^3=t_3,  t_1^2=1,  t_3^2=1;  \left(t_1,t_3 \text{-central}\right) \right\rangle.
\]
 \end{ex}

 \begin{ex}\label{ex3.2}  In this second example we consider the infinite dihedral group, given by the following
consistent polycyclic presentation:
 \[
 D_\infty = \left\langle g_1, g_2{\cal j} g_1^2=1,\ g_1^{-1}g_2g_1=g_2^{-1}\right\rangle.
 \]
 From this we get, for an arbitrary $q \geqslant 2$,
 \[
 E_q(D_\infty)=\left\langle g_1,g_2,t_1,t_2{\cal j}g_1^2=t_1, g_1^{-1}g_2g_1=g_2^{-1}t_2,
 t_1^q=1, t_2^q=1; \left( t_1,t_2 \text{-central} \right) \right\rangle.
 \]
 Checking these relations for consistency we find that this presentation is consistent.
 \end{ex}

 Now, from the polycyclic presentation of $E_q(G)$ given earlier we can determine a presentation for
$\mathfrak{E}^q(G)$.

 \begin{lem}\label{lem2} Let $\varsigma: F_m \rightarrow E_q(G)$ given by $(f_i)\varsigma = w_i,$ for $1\leq i \leq m$,
where as before $w_i = w_i(g_1, \ldots, g_n)$ is a word in the generators $g_1, \ldots, g_n.$
Then,
 	\begin{itemize}
 		\item[(i)] $Ker(\varsigma)=[K,F_m]K^q$;
 		\item[(ii)] $\mathfrak{E}^q(G) \cong Im(\varsigma)/(S)\varsigma$.
 	\end{itemize}
 	\begin{proof}

 		(i). Notice that by definition $Im(\varsigma)$ covers $G \cong E_q(G)/T_q$ and hence
$F_m/Ker(\varsigma)$ is a $q-$central extension of $G=F_m/K$. Thus, $[K,F_m]K^q \leq Ker(\varsigma)$. On the other hand,
 $F_m/[K,F_m]K^q$ is a polycyclic $q$-central extension of $G$ and, since by construction $E_q(G)$ is the largest
 $q$-central extension of $G$ with this property (by Proposition~\ref{qiso} it is a projective $q$-crossed $G-$module;
see also \cite[Lemma 3]{EN}), it follows that $E_q(G)$ contains
 $F_m/[K,F_m]K^q$ as a sub-factor via $\varsigma$. Thus, $Ker(\varsigma)=[F_m,K]K^q$. \\
 		(ii). Now, by part $(i)$ we get that $Im(\varsigma) \cong \dfrac{F_m}{[F_m,K]K^q}$ and, by definition,
$\mathfrak{E}^q(G)=\dfrac{F_m}{S[F_m,K] K^q}$. But $(S)\varsigma=\dfrac{S[K,F_m]K^q}{[K,F_m]K^q}$;
consequently, $\dfrac{Im(\varsigma)}{(S)\varsigma}=\mathfrak{E}^q(G)$.
 	\end{proof}
 \end{lem}

 Tuned in this way, in order to determine a presentation of $\mathfrak{E}^q(G)$ it suffices to determine generators for
the subgroups $Im(\varsigma)$ and $(S)\varsigma$ of $E_q(G)$, since standard methods for polycyclic groups can be used
in order to construct a consistent polycyclic presentation for the quotient $Im(\varsigma)/(S)\varsigma$ (see
also \cite[Chap. 8]{Holt}).

Certainly, a set of generators for $Im(\varsigma)$ is given by $w_1,...,w_m$. Let $s_1,...,s_k$ be a set of defining
relators for the finitely presented group $H=F_m/S$. Then, $(S)\varsigma$ is generated by
$(s_1)\varsigma,...,(s_k)\varsigma$ as a subgroup, once
$(S)\varsigma \leq T_q$ is central in $E_q(G)$. Thus, a set of generators for $(S)\varsigma$ can be determined by
evaluating the relators $s_1,...,s_k$ in $E_q(G)$.

 \section{Polycyclic presentations for the $q$-exterior square $G \wedge^q G$ and for the second homology group
$H_2(G, \Z_q)$} \label{sec:4}

According to Lemmas~\ref{lem:liv} and~\ref{lem1}, we have the following:
 \begin{cor}\label{cor1} $(i)$ $G \wedge^q G \cong E_q(G)' E_q(G)^q$;\\
\phantom{COROLLARY 4.1} $(ii)$ $H_2(G,\Z_q) \cong \left(E_q(G)'E_q(G)^q \right) \cap T_q.$
 \end{cor}

 Therefore, in order to obtain a presentation for the groups $G \wedge^q G$ and $H_2(G, \Z_q)$, for a
polycyclic group $G$ given by a consistent polycyclic presentation, we apply standard methods to determine
presentations of subgroups of polycyclic groups (see for instance \cite[Chap. 8]{Holt}).

 By the isomorphism given in  Corollary~\ref{cor1}, we obtain generators for $G \wedge^q G$ via $E_q(G)$.
 \begin{prop} The subgroup $(E_q(G))'(E_q(G))^q$ of $E_q(G)$ is generated by the set
\[
\left\langle [g_i, g_j], g_k^q \vert 1\leq i< j \leq, 1\leq k \leq n \right\rangle .
\]
 	\begin{proof} As
 	$E_q(G)=\left\langle g_1,...,g_n,t_1,...,t_l \vert r_i=t_i, \left(t_i\text{ }q\text{-central}
 	\right)\right\rangle$
 	is polycyclic, we immediately get that
 	$$(E_q(G))'= \left\langle [g_i, g_j] \vert 1\leq i<j \leq n \right\rangle.$$
Now by a simple induction on $q$ we see that each power $g^q, g \in E_q(G),$ is a word in the commutators
$[g_i, g_j]$ and in the $q$-th powers $g_i^q$ of the generators of $G$.
 	\end{proof}
 \end{prop}

Let $\pi: E_q(G) \to G$ be the natural epimorphism and choose a preimage $\widetilde{g} \in E_q(G)$ for each
$g \in G$. By Proposition~\ref{qiso} we have an isomorphism $\beta: G \wedge^q G \to  (E_q(G))'(E_q(G))^q$ such that
$(g \wedge h)^\beta =  [\widetilde{g}, \widetilde{h}]$ and $(\widehat{k})^\beta = (\widetilde{k})^q.$

\begin{rem} \label{obs9} \hspace{7.9cm}
	\begin{itemize}
		\item[$(i)$] \label{obs9} As we have seen in the Introduction, $G$ acts naturally on $G \wedge^q G$
via $(g \wedge h)^k = g^k \wedge h^k$, $(\widehat{k})^g = \widehat{k^g}$ for all $g, h, k \in G$. In addition, this
action
is compatible with the isomorphism $\beta$  and $(g \wedge h)^k$ corresponds to
$[\widetilde{g}, \widetilde{h}]^k = [\widetilde{g^k}, \widetilde{h^k}]$, while $(\widehat{k})^g$ corresponds
to $(\widetilde{k}^q)^g = \widetilde{(k^g)}^q$. The image $w^k$ of an arbitrary element $w \in G \wedge^q G$ is
obtained by writing $w$ as a product of $q$-th powers and commutators and then computing the action of $k$ upon each
factor.
		\item[$(ii)$] By construction, the map $\lambda: G \times G \times G \to G \wedge^q G, \;
		(g, h, k) \mapsto [\widetilde{g}, \widetilde{h}] (\widetilde{k})^q$ is a $q$-biderivation. Applying
$\beta$ it corresponds to the
$q$-biderivation $\lambda: G \times G \times G \to (E_q(G))'(E_q(G))^q,  (g, h, k) \mapsto (g \wedge h) \widehat{k}$.
	\end{itemize}
\end{rem}

Note that we can determine the image of the action of $G$ and of the $q$-biderivation $\lambda$ (see \ref{qbide}) in
the polycyclic presentation of $G \wedge^q G$ by using the above Remark.

\begin{ex}[Continuation of Example~\ref{ex3.1}] We determine $S_3 \wedge^2 S_3$
by identifying it with the subgroup
$\left(E_2(S_3)\right)' \left(E_2(S_3)\right)^2 = \left\langle [g_1, g_2], g_1^2, g_2^2 \right\rangle =
\left\langle w \vert w^6 \right\rangle \cong C_6,$ where $w = g_2^2t_1$.
	\begin{itemize}
		\item[(i)] By Remark~\ref{obs9} the image of $(g_1 \wedge g_2)^{g_1}$ in the consistent polycyclic
presentation of $S_3 \wedge^2 S_3$ corresponds to the element
		$[g_1, g_2]^{g_1}$ of $\left(E_2(S_3)\right)'\left(E_2(S_3)\right)^2$.
		In turn, evaluating this element using the relations of $\left(E_2(S_3)\right)'\left(E_2(S_3)\right)^2$
		 we obtain the element
		\[
		g_1^{-1}[g_1, g_2]g_1 = g_1^{-1}g_2^2g_1 = g_2^2 = w^4.
		\]
		\item[(ii)] Analogously, the image of $(g_1, g_2, g_1)\lambda$ in the consistent polycyclic
presentation of $S_3 \wedge^2 S_3$ corresponds to the element
$[g_1, g_2] g_1^2$ of $\left(E_2(S_3) \right)' \left(E_2(S_3) \right)^2$ and thus,
using the relations of $\left(E_2(S_3)\right)'\left(E_2(S_3)\right)^2$, we evaluate this element to get
 \[
 [g_1, g_2] g_1^2 = g_2^2 t_1 = w.
 \]
	\end{itemize}
\end{ex}

\begin{ex}[Continuation of Example~\ref{ex3.2}] Now we determine $D_\infty \wedge^2 D_\infty$ by identifying
it with the subgroup
\begin{align*}
\left(E_q(D_\infty)\right)'\left(E_q(D_\infty)\right)^2 & =  \left\langle [g_1,g_2],g_1^2, g_2^2\right\rangle \\
  \  & = \left\langle w_1, w_2, w_3 \vert w_1^2, [w_1,w_2], [w_1,w_3], w_2^2, [w_2,w_3]\right\rangle \\
  \ & \cong C_2\times C_2 \times C_\infty
\end{align*}
where $w_1 = g_2^2, w_2 = t_1, w_3 = t_2$.
\begin{itemize}
\item[(i)] Analogous to the previous example, by Remark~\ref{obs9} the image of $(g_1 \wedge g_2)^{g_1}$ in the
consistent polycyclic presentation of $D_\infty \wedge^2 D_\infty$ corresponds to the element
$[g_1, g_2]^{g_1}$ in $\left(E_q(D_\infty)\right)' \left(E_q(D_\infty)\right)^2,$ that is, to
\[
g_1^{-1}[g_1, g_2]g_1 = g_1^{-1}g_2^2 t_2 g_1 = g_2^{-1} t_2 = w_1^{-1} w_3.
\]
\item[(ii)] Similarly, the image of $(g_1, g_2, g_1)\lambda$
corresponds to the element $[g_1, g_2] g_1^2$ of $\left(E_q(D_\infty)\right)' \left(E_q(D_\infty)\right)^2$, which
results
in
\[
[g_1, g_2]g_1^2 = g_2^2t_2t_1 = w_1 w_2 w_3.
\]
	\end{itemize}
\end{ex}
\section{The $q$-exterior center of a polycyclic group} \label{sec4}

Our next step is to show that we can easily determine  the
$q$-exterior center of a polycyclic group $G$  given by a consistent polycyclic presentation, using a consistent
polycyclic presentation
for $E_q(G)$ and standard methods for polycyclic groups (see \cite[Chap. 8]{Holt}). These techniques also extend those
found in \cite{EN} for the
exterior center (case $q = 0$).

\begin{thm} \label{qapresentext}
Let $G$ be a polycyclic group and $\pi:E_q(G) \to G$ the natural epimorphism. Then,
$Z_q^\wedge(G)=\left(Z(E_q(G))\right)\pi$.
\end{thm}
\begin{proof} For each $g\in G$ let $\widetilde{g}$ be a pre-image of $g$ in $E_q(G)$ under the epimorphism
$\pi$, i.e., $(\widetilde{g})\pi = g$. Now, $[\widetilde{g}, \widetilde{a}] = 1$ for all $a \in G$ if, and only if,
$[\widetilde{g}, x]=1$ for all $x \in E_q(G)$. Indeed, given $x \in E_q(G)$ then $x^\pi \in G$ and by assumption
$[\widetilde{g},\widetilde{x^\pi}] = 1$. On the other hand,
$\left(\widetilde{x^\pi}\right)^\pi=x^\pi$.
	Thus, $\left(\widetilde{x^\pi}\right)^{-1}x \in Ker(\pi) \leq Z(E_q(G))$.
	Therefore, we have \\
		$1=[\widetilde{x^\pi}^{-1}x,\widetilde{g}]$
		 $=[\widetilde{x^\pi}^{-1},\widetilde{g}]^x[x,\widetilde{g}]$
		 $=[\widetilde{x^\pi},\widetilde{g}]^{-(\widetilde{x^\pi})^{-1}x}[x,\widetilde{g}]$
		 $=[\widetilde{x^\pi},\widetilde{g}][x,\widetilde{g}]$
		 $=[x,\widetilde{g}]$.

		Conversely, given $a \in G$ we have $\widetilde{a} \in E_q(G)$.
		By assumption $[x,\widetilde{g}]=1$ for all
		$x \in  E_q(G)$ and, in particular, for $x=\widetilde{a}$. Thus,
		$[\widetilde{a},\widetilde{g}]=1$, for all $a\in G$.

		Now, remind that the map $\beta: G\wedge^q G \to (E_q(G))'(E_q(G))^q$ given by
$(g \wedge h)\beta=[\widetilde{g},\widetilde{h}]$ and $(\widehat{k})\beta=\widetilde{k}^q$ is an isomorphism, and thus
we get:
		\begin{itemize}\item[] $Z_q^\wedge(G)=\left\{g\in G\vert
		1=g\wedge a \in G\wedge^qG \hspace{0.2cm}, \; \forall \; a\in G \right\} $
			\item[] \hspace{1.4cm}$=\left\{g\in G\vert [\widetilde{g}, \widetilde{a}]=1
			\hspace{0.2cm}, \; \forall \; a \in G\right\}$ (by using $\beta$)
			\item[] \hspace{1.4cm}$=\left\{g\in G\vert [\widetilde{g},x]=1,  \;
			\forall \; x \in E_q(G)\right \}$
			\item[] \hspace{1.4cm}$=\left\{g\in G\vert \widetilde{g} \in Z(E_q(G))\right\}$
			\item[] \hspace{1.4cm}$=\left(Z(E_q(G))\right)\pi$.
		\end{itemize} \end{proof}

	Thus, by Theorem~\ref{qapresentext}, the $q$-exterior center of a group $G$ given by a consistent polycyclic
presentation can be easily determined: first we determine a polycyclic presentation for $E_q(G)$ and its corresponding
natural epimorphism $\pi : E_q(G) \to G.$ Then we compute the center $Z(E_q(G))$ using standard methods for
polycyclically presented groups (see \cite[Chap. 8]{holt}) and, finally, we apply $\pi$ to obtain
$Z_q^\wedge(G) = \left(Z(E_q(G))\right)\pi$.

	\begin{ex}[Continuation of Example~\ref{ex3.1}] It follows from the consistent polycyclic presentation of
	$E_2(S_3)$ that  $Z \left(E_2(S_3)\right) = \left\langle t_1, t_3\right\rangle$.
	Thus,  $Z_2^\wedge\left(S_3\right) = 1$ and so, as one should expect, $S_3$ is $2$-capable.
	In fact, the group $Q$ given by
	$$Q = \left\langle a, b \mid a^4 = 1, a^{-1} b a = b^{-1}, b^3 = 1 \right\rangle$$
	has center $Z(Q) = <a^2> = Z_2(Q),$ of order 2, and $S_3 \cong Q/Z(Q)$.
	\end{ex}

	\begin{ex}[Continuation of Example~\ref{ex3.2}] It follows from the polycyclic presentation of
	$E_q(D_\infty)$ that $Z_q\left(E_q(D_\infty)\right) = \left\langle t_1, t_2 \right\rangle$.
Thus,  $Z^\wedge_q\left(D_\infty\right) = 1;$ hence,  $D_\infty$ is $q$-capable for all $q \geq 0$.
Indeed, the group $$Q = \left\langle a, b \mid a^{2q} = 1, a^{-1} b a = b^{-1}\right\rangle$$
has center $Z(Q)=<a^2> = Z_q(Q),$ of order $q$ and $D_\infty \cong Q/Z(Q)$.
	\end{ex}


\section{A consistent polycyclic presentation for
$\nu^q(G)/\Delta^q(G)$} \label{sec:5}

As seen in Sec.~\ref{sec:4}, we can determine a consistent polycyclic presentation $F_r/U$ for the $q$-exterior square
$G \wedge^q G$ in the generator $w_1,...,w_r$ and relators, say $u_1,...,u_s$. From such a presentation we will
determine a consistent polycyclic presentation for the group $\nu^q(G)/\Delta^q(G)$.
Remember that $\nu^q(G)/\Delta^q(G) \cong (G \wedge^q G) \rtimes (G \times G)$.

According to Remark~\ref{obs9}, we can determine the image of the $q$-biderivation
 $\lambda: G \times G \times G \to G \wedge^q G$: $(g,h,1) \mapsto (g \wedge h)$ and $(1,1,k) \mapsto \widehat{k}$
 in the consistent polycyclic presentation we obtained for $G \wedge^q G$. Analogously, we can construct the
 natural action of $G$ on the presentation found for $G \wedge^q G$, which is given by
$(g \wedge h)^x = g^x \wedge h^x$, $(\widehat{k})^x = \widehat{(k^x)}$.

Recall that we are given a consistent polycyclic presentation of group $G$; as before,  \\ $G = \langle g_1, \ldots,
g_n \mid r_1, \ldots, r_l \rangle.$
\begin{define} \label{deftauq}
Define $\tau^q(G)$ \index{$\tau^q(G)$} to be the group generated by
$g_1,...,g_n, g_1^\varphi,..., g_n^\varphi, w_1,...,w_r$, subject to the following defining relations:
\begin{itemize}
\item[(1)] $r_i(g_1,...,g_n) = 1$, for  $1\leq i \leq l$;
\item[(2)] $r_i(g_1^\varphi,...,g_n^\varphi) = 1$, for $1\leq i \leq l$;
\item[(3)] $u_i(w_1,...,w_r)=1, $ for $1\leq i \leq s$;
\item[(4)] $g_i^{-1} g_j^\varphi g_i = g_j^\varphi \left((g_i, g_j, 1)\lambda \right)^{-1}$, for $1\leq i,j \leq n$,\\
	$g_i g_j^\varphi g_i^{-1} = g_j^\varphi \left((g_i^{-1},g_j,1) \lambda\right)^{-1},$ for $1\leq i,j \leq n$
$i\notin I;$
\item[(5)] $g_j^{-1} w_i g_j = w_i^{g_j}$, for $1\leq i \leq r$, $1\leq j\leq n,$\\
      $g_jw_ig_j^{-1}=w_i^{g_j^{-1}}$, for $1\leq i \leq r$, $1\leq j\leq n$, $j\notin I$,\\
      $g_j^{-\varphi}w_ig_j^\varphi=w_i^{g_j}$,for $1\leq i \leq r$, $1\leq j\leq n$,\\
      $g_j^\varphi w_ig_j^{-\varphi}=w_i^{g_j^{-1}}$, for $1\leq i \leq r$, $1\leq j\leq n$, $j\notin I$.
\end{itemize}
\end{define}

	Notice that we can compute the right hand side of the relations $(4)$ and $(5)$
as words in $w_1,...,w_r$ (see Remark~\ref{obs9}).

	\begin{thm} \label{teo14}
	Let $W  \leq \tau^q(G)$ be the subgroup $\langle w_1,...,w_r \rangle$.  Then we have:
		\begin{itemize}
			\item[(i)] $W$ is a normal subgroup of $\tau^q(G)$ and
            $\tau^q(G)/W \cong G \times G$;
			\item[(ii)] The presentation of $\tau^q(G)$ in definition~\ref{deftauq} is a consistent
polycyclic presentation;
		\item[(iii)] $W \cong G \wedge^q G$;
			\item[(iv)] The map
            $\psi: \nu^q(G) \to \tau^q(G)$ defined by
			$(g_i)\psi = g_i$, $(g_i^\varphi)\psi = g_i^\varphi$ and $(\widehat{k})\psi=(1,1,k)\lambda$,
			for all $1 \leq i \leq n$ and all $k\in G,$
			extends to a well defined homomorphism (also denoted by $\psi$) such that
			$\Ker \psi = \Delta^q(G)$.
		\end{itemize}
		\begin{proof} The proof is mainly  based on a careful analysis of the sets of defining relations
		(1) -- (5) of $\tau^q(G)$, as established in the Definition~\ref{deftauq}.
		The relations $(5)$ tell us that $W$ is in fact a normal subgroup of $\tau^q(G)$. The relations
$(1), (2)$ and $(4)$ imply that $\tau^q(G)/W \cong G \times G^\varphi$. In addition, relations $(3)$ show that
$W$ is a factor of $G \wedge^q G$. Thus, $\tau^q(G)$ satisfies the exact sequence
\[
G \wedge^q G \to \tau^q(G)\to G \times G^\varphi \to 1.
\]
Now, the relations  $(5)$ imply that $G \times G^\varphi$ acts by conjugation on $W,$
in the same way as $G \times G$ acts naturally on $G \wedge^q G$.
In particular, we get that $[w, g] = w^{-1}w^g$ and, analogously,
$[w, h^\varphi]= w^{-1} w^{h^\varphi}$, for all words $w$ in $w_1,...,w_r$, all words $g$ in $g_1,...,g_n$  and all
words
$h^\varphi$ in $g_1^\varphi,...,g_n^\varphi$. Furthermore, the definition of a $q$-biderivation and relations
$(4)$ imply that $[g, h^\varphi] = (g,h,1)\lambda$, for all words $g$ in $g_1,...,g_n$ and
$h^\varphi$ in $g_1^\varphi,...,g_n^\varphi$.\\
Part $(i)$ then follows directly from the above considerations. \\
$(ii).$ The relations $(1)$ -- $(5)$ already have the form of a polycyclic presentation.
Thus, it remains to check them for consistency. Well, all consistency relation in the generators
$g_1,...,g_n$ is satisfied, once relations $(1)$ come from a consistent polycyclic presentation of $G$. Analogously,
for the relations $(2)$ and $(3)$; they say that all consistency in the generators $g_1^\varphi,...,g_n^\varphi$ and
$w_1,...,w_r$ are also satisfied. Besides that, if a consistency relation
involves one generator of the $w_1,...,w_r$, then it is satisfied, once $G \times G^\varphi$ acts on $W$ likewise
$G \times G$ acts naturally on $G \wedge^q G$. Therefore, the bottom line is really to check the
consistency relations in
$g_1,...,g_n, g_1^\varphi,..., g_n^\varphi$, involving mixed generators $g_i$ e $g_j^\varphi$.
They are:
\begin{description}
	\item \hspace{4.5cm}$g_k^\varphi(g_jg_i)=(g_k^\varphi g_j)g_i$ para $j>i$;
	\item \hspace{4.5cm}$g_k^\varphi(g_j^\varphi g_i)=(g_k^\varphi g_j^\varphi)g_i$ para $k>j$;
	\item \hspace{4.5cm}$((g_j^\varphi)^{e_j})g_i=(g_j^\varphi)^{e_j-1}(g_j^\varphi g_i)$ para $j\in I$;
	\item \hspace{4.5cm}$g_j^\varphi(g_i^{e_i})=(g_j^\varphi g_i)g_i^{e_i-1}$ para $i\in I$;
	\item \hspace{4.5cm}$g_j^\varphi=(g_j^\varphi g_i^{-1})g_i$ para $i\notin I$.
\end{description}
Consider for example the first of these relations.
Supposing that $g_i^{-1} g_j g_i = r_{ij}(g_1,...,g_n) = r_{ij}$ in the defining relations of $G$ and using the fact
that $\lambda$ is a $q$-biderivation, we get:
\begin{description}
	\item \hspace{4.5cm}$g_k^\varphi(g_jg_i)=(g_jg_i)(g_k^\varphi)^{g_jg_i}$
	\item \hspace{6.1cm}$=(g_ir_{ij})\left(g_k^\varphi\left(\left(g_jg_i,g_k,1\right)\lambda\right)^{-1}\right)$,
\item \hspace{4.5cm}$(g_k^\varphi g_j)g_i=\left(g_jg_k^\varphi\left(\left(g_j,g_k,1\right)\lambda\right)^{-1}\right)g_i$
\item \hspace{6.1cm}$=g_jg_k^\varphi g_i\left(\left(\left(g_j,g_k,1\right)\lambda\right)^{-1}\right)^{g_i}$
\item \hspace{6.1cm}$=g_jg_i(g_k^\varphi)^{g_i}\left(\left(g_j,g_k,1\right)\lambda^{g_i}\right)^{-1}$
\item \hspace{6.1cm}$=g_ir_{ij}g_k^\varphi \left(\left(g_i,g_k,1\right)\lambda\right)^{-1}\left(\left(
g_j^{g_i},g_k^{g_i},1\right)\lambda\right)^{-1}$
\item \hspace{6.1cm}$=g_ir_{ij}g_k^\varphi \left(\left(g_j^{g_i},g_k^{g_i},1\right)\lambda\left(
g_i,g_k,1\right)\lambda\right)^{-1}$
\item \hspace{6.1cm}$=g_ir_{ij}g_k^\varphi \left(\left(g_jg_i,g_k,1\right)\lambda\right)^{-1}$.
\end{description}

The other consistency relations can be checked by similar calculations.
Thus we obtain that $\tau^q(G)$ is given by a consistent polycyclic presentation.\\
Part $(iii)$ follows from $(ii)$ and from the theory of polycyclic presentations
(see for instance \cite[Sec. 8.3]{Holt}), once $W$, as a subgroup of $\tau^q(G)$, has a consistent polycyclic
presentation in the generators $w_1,...,w_r$ and relations $u_1,...,u_s$. Therefore, $W \cong G \wedge^q G$.\\
$iv)$ Since $\lambda$ is a $q$-biderivation, all relations of $\nu^q(G)$ hold in $\tau^q(G)$.
Thus, $\psi : \nu^q(G) \to \tau^q(G)$ is an epimorphism, once $g_i, g_j^\varphi \in Im(\psi)$ for all
$1\leq i,j \leq n$, and if a word $w_i \in W \cong G\wedge^qG$ is a product of commutators and
$q$-th powers, then $w_i \in Im(\psi)$, for all $1 \leq i \leq r$.
Consequently, $Im(\psi)=\tau^q(G)$. Besides that,
$([g, h^\varphi])\psi=[g, h^\varphi]=(g,h,1)\lambda$ and
$(\widehat{k})\psi=(1,1,k)\lambda$, for all words, $g$ in generators the $g_1,...,g_n$,
$h^\varphi$ in $g_1^\varphi,...,g_n^\varphi$, and all $k \in G$.
Therefore, the map induced by $\psi$ on the subgroup $\Upsilon^q(G)$ coincides with
the map $\delta: \Upsilon^q(G) \to G \wedge^q G$, by construction. We then get the following commutative diagram:

\xymatrix{&&& 1\ar[r] & \Upsilon^q(G)\ar[r]\ar[d]^\delta &
\nu^q(G)\ar[r]\ar[d]^\psi & G\times G^\varphi \ar[r]\ar[d]^\cong & 1 \\
 &&& 1\ar[r] &  W\ar[r] & \tau^q(G)\ar[r] & G\times G^\varphi \ar[r] & 1.}

It follows that $Ker(\psi)\leq \Upsilon^q(G)$ and, consequently, $Ker(\psi)=Ker(\delta)=\Delta^q(G)$. This completes
the proof.
		\end{proof}
	\end{thm}

\begin{ex}[Continuation of Example~\ref{ex3.1}] According to the above result, the following polycyclic presentation is
a presentation of
  $\tau^2(S_3)=\nu^2(S_3)/\Delta^2(S_3)$ as the group generated by
$g_1, g_2, g_1^\varphi, g_2^\varphi, w $ subject to the following relations:
\begin{itemize}
	\item[(1)] $g_1^2=1$, $g_1^{-1}g_2g_1=g_2^{-1}$, $g_2^3$,
	\item[(2)] $(g_1^\varphi)^2=1$, $(g_1^\varphi)^{-1}g_2^\varphi g_1^\varphi=(g_2^\varphi)^{-1}$,
	$(g_2^\varphi)^3$,
	\item[(3)] $w^6=1$,
	\item[(4)] $g_1^{-1}g_1^\varphi g_1=g_1^\varphi$, \\
	$g_1^{-1}g_2^\varphi g_1=g_2^\varphi w^2$,  \\
	$g_2^{-1}g_1^\varphi g_2=g_1^\varphi w^4$, \\
	$g_2^{-1}g_2^\varphi g_2=g_2^\varphi$,
	\item[(5)] $g_1^{-1}w g_1=w^5$, \\
			$g_2^{-1}w g_2=w$, \\
			$(g_1^\varphi)^{-1}w g_1^\varphi=w^5$, \\
			$(g_2^\varphi)^{-1}w g_2\varphi=w$.
		\end{itemize}
	\end{ex}

	\begin{ex}[Continuation of Example~\ref{ex3.2}]  Again, according to Theorem~\ref{teo14} we find that
	$\tau^2(D_\infty)=\nu^2(D_\infty)/\Delta^2(D_\infty)$ has the polycyclic presentation in the generators
	$g_1, g_2, g_1^\varphi, g_2^\varphi, w_1, w_2, w_3 $ subject to the relations:
		\begin{itemize}
			\item[(1)] $g_1^2=1$, $g_1^{-1}g_2g_1=g_2^{-1}$,
			\item[(2)] $(g_1^\varphi)^2=1$, $(g_1^\varphi)^{-1}g_2^\varphi g_1^\varphi=(g_2^\varphi)^{-1}$,
			\item[(3)] $w_1^{-1}w_2w_1=w_2$, $w_2^2$,\\
			$w_1^{-1}w_3w_1=w_3$, $w_3^2$,\\
			$w_2^{-1}w_3w_2=w_3$,
			\item[(4)] $g_1^{-1}g_1^\varphi g_1=g_1^\varphi$, \\
			$g_1^{-1}g_2^\varphi g_1=g_2^\varphi w_1^{-1}w_3$,  \\
			$g_2^{-1}g_1^\varphi g_2=g_1^\varphi w_1w_3$, \\
			$g_2^{-1}g_2^\varphi g_2=g_2^\varphi$,\\
			$g_2g_1^\varphi g_2^{-1}=g_1^\varphi w_1^{-1}w_3$, \\
			$g_2g_2^\varphi g_2^{-1}=g_2^\varphi$,
			\item[(5)] $g_1^{-1}w_1 g_1=w_1^{-1}$, \hspace{1cm}	$g_1^{-1}w_2 g_1=w_2$,
			\hspace{0.5cm}$g_1^{-1}w_1 g_1=w_3$, \\
	$g_2^{-1}w_1 g_2=w_1$, \hspace{1.27cm} $g_2^{-1}w_2 g_2=w_2$,\hspace{0.5cm} $g_2^{-1}w_3 g_2=w_3$, \\
	$g_2w_1 g_2^{-1}=w_1$,\hspace{1.37cm} $g_2w_2 g_2^{-1}=w_2$, \hspace{0.5cm}$g_2w_3 g_2^{-1}=w_3$, \\
$(g_1^\varphi)^{-1}w_1 g_1^\varphi=w_1^{-1}$, \hspace{0.5cm}	$(g_1^\varphi)^{-1}w_2 g_1^\varphi=w_2$,
\hspace{0.6cm} $(g_1^\varphi)^{-1}w_1 g_1^\varphi=w_3$, \\
	$(g_2^\varphi)^{-1}w_1 g_2^\varphi=w_1$, \hspace{0.82cm}$(g_2^\varphi)^{-1}w_2 (g_2^\varphi)=w_2$,
	\hspace{0.5cm} $(g_2^\varphi)^{-1}w_3 (g_2^\varphi)=w_3$, \\
	$g_2^\varphi w_1 (g_2^\varphi)^{-1}=w_1$,\hspace{0.82cm} $g_2^\varphi w_2 (g_2^\varphi)^{-1}=w_2$,
	\hspace{0.8cm}$g_2^\varphi w_3 (g_2^\varphi)^{-1}=w_3$.
		\end{itemize}
	\end{ex}

	\section{A polycyclic presentation for
	$\nu^q(G)$} \label{nuqfinito}

	\hspace{0.6cm} We can now use the consistent polycyclic presentation of $\tau^q(G)$ in place of $F_n/R$ and the
	finite presentation of $\nu^q(G)$ in place of $F_m/S$. Note that the epimorphism $\psi:\nu^q(G)\to\tau^q(G)$ has
the required form. Thus we get the following

	\begin{thm}\label{teo16} $\nu^q(G) = \mathfrak{E}^q(\tau^q(G))$.
		\begin{proof} We have  $\tau^q(G)=F_n/R$,  $\nu^q(G)=F_m/S$ and the epimorphism
$\psi:\nu^q(G)\to\tau^q(G)$ with the kernel $Ker(\psi)=K/S$. By definition,
$\mathfrak{E}^q(\tau^q(G))\cong F_m/K^q[K,F_m]S$. By Theorem~\ref{teo14} the
group $\nu^q(G)$ is a $q$-central extension of $\tau^q(G)$. Thus,  $[K,F_m]K^q \leq S$.
Since by Lemma~\ref{lem2}
		$\mathfrak{E}^q(\tau^q(G))=F_m/K^q[K,F_m]S$, it follows that
$\mathfrak{E}^q(\tau^q(G))=F_m/S=\nu^q(G),$ as desired.
		\end{proof}
	\end{thm}

	Notice that we used a finite presentation for $\nu^q(G)$, to which we guarantee finiteness only in the case
	that $G$ is finite. On the other hand, the epimorphism $\psi:\nu^q(G) \to \tau^q(G)$ does not depend on the
finiteness of the presentation of $\nu^q(G)$ and so we can consider that epimorphism. If $G$ is an infinite polycyclic
group then by definition $\nu^q(G)$ is given by an infinite presentation, say $F/S$, where  $F$ is a free group on the
generators of $\nu^q(G)$, which we denote by $X$, of infinite rank, and where
$S$ is the normal closure of the relations
\eqref{RR1}---\eqref{RR6}. On the other hand,
$\nu^q(G)$
is polycyclic and so it has a finite polycyclic presentation
$\left\langle X_0 \vert S_0 \right\rangle$, where $X_0 \subseteq X$ and $S$ is the normal closure of
 $S_0$, $\bar{S_0}=S$. Thus, we can use the results in Lemma~\ref{lem2} and it suffices to prove that the
image of $\varsigma$ is generated by the elements
$g_1,...,g_n,g_1^\varphi,...,g_n^\varphi,
\widehat{g_i},...,\widehat{g_n},$
and  $(S)\varsigma$ is generated by the defining relations of $\nu^q(G)$ evaluated only on the polycyclic
generators of $G$.

	\begin{prop} Consider the subgroup $L$ of $E_q\left(\tau^q(G)\right)$ given by
	\[
	L = \left\langle g_1,...,g_n, g_1^\varphi,...,g_n^\varphi, (1,1,g_1)\lambda,...,(1,1,g_n)\lambda \right\rangle.
	\]
	Then, $Im(\varsigma)=L$. In addition,
$(S)\varsigma$ is generated by the defining relations of $\nu^q(G)$ in the polycyclic generators of $G$ and
$G^\varphi$ in $E_q\left(\tau^q(G)\right)$.


		\begin{proof} By definition of $L$, to show that $Im(\varsigma) = L$ it suffices to show that
$(1,1,k)\lambda \in L$ for all $k \in G$. Let's prove this by induction on the number of polycyclic generators of $G$.
		If $n=1$ then $G=<g_1>$ and so $k=g_1^\alpha$, for some  $\alpha \in \Z$. Thus,
		$(1,1,k) \lambda=(1,1,g_1^\alpha)\lambda$. For $\alpha \geq 2$, using relation~\ref{relqbide4} in
Definition~\ref{def:qbide} we have
\begin{description}
\item[] $(1,1,k)\lambda=(1,1,g_1^2)\lambda$
\item[]\hspace{1.6cm} $=\underbrace{(1,1,g_1)\lambda}_{\in L}
\underbrace{\displaystyle\prod_{i=1}^{q-1}\left(\left((g_1,g_1^{-i},1)\lambda\right)^{g_1^{q-1-i}}\right)}_{=[g_1,g_1^{-
i\varphi}] \in L}\underbrace{(1,1,g_1)\lambda}_{\in L}$ $\in L$.
\end{description}
By assuming it for $\alpha-1$ then, analogously,
\begin{description}
	\item[] $(1,1,k)\lambda=(1,1,g_1^\alpha)\lambda$
	\item[]\hspace{1.6cm} $=(1,1,g_1)\lambda \displaystyle\prod_{i=1}^{q-1}\left(\left((g_1,g_1^{-i(n-1)},1)
	\lambda\right)^{g_1^{q-1-i}}\right) (1,1,g_1^{n-1})\lambda$ $\in L$.
\end{description}
In addition, again by the very definition of $\lambda$, as above, we obtain $(1,1,1)\lambda = 1$ and
$(1,1,k^{-1})\lambda = \left( \prod_{i=1}^{q-1}\left((x,x^{i},1)\lambda \right)\right)^{-1}
\left((1,1,x)\lambda \right)^{-1},$ which are elements of $L$. This completes the case $n = 1$.

Suppose $n \geq 1$ and that our assertion is true for $n-1$. If $k = g_1^{\alpha_1}...g_n^{\alpha_n}$ then, the same
argument used above gives:
\begin{description}
\item[] $(1,1,k)\lambda=(1,1,g_1^{\alpha_1}...g_n^{\alpha_n})\lambda$
\item[]$=(1,1,g_1^{\alpha_1})\lambda \displaystyle\prod_{i=1}^{q-1}
\left(\left((g_1^{\alpha_1},(g_2^{\alpha_2}...g_n^{\alpha_n})^{-i},1)
\lambda\right)^{(g_1^{\alpha_1})^{q-1-i}}\right) (1,1,g_2^{\alpha_2}...g_n^{\alpha_n})\lambda$ $\in L$
\end{description}
Thus, $Im(\varsigma)=L$.
		\end{proof}
	\end{prop}

	Therefore, having obtained  a consistent polycyclic presentation of $\tau^q(G)$, we can extend it by
adding new ($q$-central) generators $t_i$, one for each relator $r_i$ of $\tau^q(G)$, and changing
	each relator $r_i$ by $r_it_i^{-1}$. Then, we evaluate the consistency relations among the relators of
$\nu^q(G)$ in this new presentation and apply Lemma \ref{lem2}(ii).

	The following result can by used in order to reduce the number of new generators added and the number of
relators evaluated in this process.

	\begin{lem} It is redundant to add new generators corresponding to relations (1) and (2) in the definition of
	$\tau^q(G)$. If these generators are not introduced, then it is redundant evaluate the relators (1) and (2) in
the definition of $\nu^q(G)$.
	\begin{proof} The relators (1) and (2) in the definition of $\nu^q(G)$ coincide with the relators (1) and (2)
in the definition of $\tau^q(G)$. Therefore, if we add new generators corresponding to those relators in (1)
and (2) of Definition~\ref{deftauq} and then we evaluate the relators (1) and (2) in the definition of $\nu^q(G)$,
then as a result we obtain the corresponding generators.
This means that the corresponding generators are eliminated in the process of the constructing the factor
group as described in Lemma~\ref{lem2}(b). This proves the result.
	\end{proof}
	\end{lem}

	\begin{ex}[Continuation of Example~\ref{ex3.1}] We compute a polycyclic presentation of $\nu^2(S_3)$ as a
central extension $\mathfrak{E}^2(\tau^2(S_3))$ of $\tau^2(S_3)$. There is a lot of calculations to get such a
presentation (by hand) and so we'll omit the details. We obtain a polycyclic presentation for $\nu^2(S_3)$ in the
generators
$g_1, g_2, g_1^\varphi, g_2^\varphi, w, t$ and defining relations given by:
\begin{itemize}
 \item[(1)] $g_1^2=1$, $g_1^{-1}g_2g_1=g_2^{-1}$, $g_2^3$,
 \item[(2)] $(g_1^\varphi)^2=1$, $(g_1^\varphi)^{-1}g_2^\varphi g_1^\varphi=(g_2^\varphi)^{-1}$, $(g_2^\varphi)^3$,
 \item[(3)] $w^6=t$, $t^2$, $t$-central,
 \item[(4)] $g_1^{-1}g_1^\varphi g_1=g_1^\varphi w^6$, \\
      $g_1^{-1}g_2^\varphi g_1=g_2^\varphi w^8$,  \\
      $g_2^{-1}g_1^\varphi g_2=g_1^\varphi w^4$, \\
      $g_2^{-1}g_2^\varphi g_2=g_2^\varphi$,
\item[(5)] $g_1^{-1}w g_1=w^5$, \\
      $g_2^{-1}w g_2=w$, \\
      $g_1^{-1}w g_1=w^5$, \\
      $g_2^{-1}w g_2=w$.
\end{itemize}
From this we get the $2$-tensor square $S_3 \otimes^2 S_3 \cong <w> \leq \nu^2(S_3)$, that is,
$S_3\otimes^2S_3 \cong \Z_{12}$.
In addition, we immediatly find that $\Delta^2(S_3) \cong \Z_2$.
	\end{ex}

\begin{ex} Here we compute a polycyclic presentation for $\nu^3(D_\infty)$ as the
group  generatoed by  $g_1, g_2$, $g_1^\varphi, g_2^\varphi$, $w_1, w_2$
subject to the relations:
\begin{itemize}
\item[(1)] $g_1^2=1$, $g_1^{-1}g_2g_1=g_2^{-1}$,
\item[(2)] $(g_1^\varphi)^2=1$, $(g_1^\varphi)^{-1}g_2^\varphi g_1^\varphi=(g_2^\varphi)^{-1}$,
\item[(3)] $w_1^2$,  $w_1^{-1}w_2w_1=w_2^{-1}$,
\item[(4)] $g_1^{-1}g_1^\varphi g_1=g_1^\varphi$, \\
	$g_1^{-1}g_2^\varphi g_1=g_2^\varphi w_2^{-2}$,  \\
	$g_2^{-1}g_1^\varphi g_2=g_1^\varphi w_2^2$, \\
	$g_2^{-1}g_2^\varphi g_2=g_2^\varphi$,\\
	$g_2g_1^\varphi g_2^{-1}=g_1^\varphi w_2^{-2}$, \\
	$g_2g_2^\varphi g_2^{-1}=g_2^\varphi$,
\item[(5)] $g_1^{-1}w_1 g_1=w_1$, \hspace{3cm}	$g_1^{-1}w_2 g_1=w_2^{-1}$, \\
	$g_2^{-1}w_1 g_2=w_1w_2^2$, \hspace{2.5cm} $g_2^{-1}w_2 g_2=w_2$,\\
	$g_2w_1 g_2^{-1}=w_1w_2^{-2}$,\hspace{2.37cm} $g_2w_2 g_2^{-1}=w_2$,\\
	$(g_1^\varphi)^{-1}w_1 g_1^\varphi=w_1$, \hspace{2.3cm}	$(g_1^\varphi)^{-1}w_2 g_1^\varphi=w_2^{-1}$,\\
	$(g_2^\varphi)^{-1}w_1 g_2^\varphi=w_1w_2^2$, \hspace{2cm}$(g_2^\varphi)^{-1}w_2 (g_2^\varphi)=w_2$,\\
	$g_2^\varphi w_1 (g_2^\varphi)^{-1}=w_1w_2^{-2}$,\hspace{1.82cm} $g_2^\varphi w_2 (g_2^\varphi)^{-1}=w_2$.
\end{itemize}
According to this presentation we find that $D_\infty \otimes^3 D_\infty \cong D_\infty$.
	\end{ex}

	Notice that the computation of a presentation of $\nu^q(G)$ becomes relatively simple if the group $G$ is
$q$-perfect, according to Theorem~\ref{thm:qperfect} bellow. We shall continue using the same notation as before. More
specifically, let $F_n/R$ be consistent polycyclic presentation for the polycyclic group $G$ in the generators
 $g_1,...,g_n$,  relators $r_1,...,r_l$, and index set $I$. Let $F_r/U$ be a consistent polycyclic presentation
for $G \wedge^q G$ in the generators $w_1,...,w_r$ and relators $u_1,...,u_s$, as found in Sec.~\ref{sec:4}.
We determine the image of the $q$-biderivation
$\lambda: G \times G \times G \to G \wedge^q G$: $(g,h,1) \mapsto (g\wedge h)$ e $(1,1,k) \mapsto \widehat{k}$
in the consistent polycyclic presentation obtained for $G \wedge^q G$ and construct the natural action of $G$
on that presentation found for $G\wedge^qG$ (as defined before:
$(g\wedge h)^x=g^x\wedge h^x$, $(\widehat{k})^x=\widehat{k^x}$, according to Remark~\ref{obs9}).

	\begin{thm} \label{thm:qperfect}
	Let $G$ be a polycyclic group given as above. If $G$ is
$q$-perfect, then the group $\nu^q(G)$ is the group generated by $g_1,...,g_n,g_1^\varphi,...,g_n^\varphi$,
$w_1...,w_r$, subject to the defining relations
\begin{itemize}
	\item[(1)] $r_i(g_1,...,g_n)=1$ for $1\leq i \leq l$,
	\item[(2)] $r_i(g_1^\varphi,...,g_n^\varphi)=1$ para $1\leq i \leq l$,
	\item[(3)] $u_i(w_1,...,w_r)=1$ para $1\leq i \leq s$,
	\item[(4)] $g_i^{-1}g_j^\varphi g_i=g_j^\varphi \left\{(g_i,g_j,1)\lambda\right\}^{-1}$ para $1\leq i,j \leq n$,\\
$g_ig_j^\varphi g_i^{-1}=g_j^\varphi \left\{(g_i^{-1},g_j,1)\lambda\right\}^{-1}$ para $1\leq i,j \leq n$ $i\notin I$
	\item[(5)] $g_j^{-1}w_ig_j=w_i^{g_j}$, para $1\leq i \leq r$, $1\leq j\leq n$\\
	$g_jw_ig_j^{-1}=w_i^{g_j^{-1}}$, para $1\leq i \leq r$, $1\leq j\leq n$, $j\notin I$\\
	$g_j^{-\varphi}w_ig_j^\varphi=w_i^{g_j}$, para $1\leq i \leq r$, $1\leq j\leq n$\\
	$g_j^\varphi w_ig_j^{-\varphi}=w_i^{g_j^{-1}}$, para $1\leq i \leq r$, $1\leq j\leq n$, $j\notin I$.
\end{itemize}
	\begin{proof} In effect, according to Definition~\ref{deftauq} the above presentation is the same as that of
$\tau^q(G)$. By Theorem~\ref{teo14},  $\tau^q(G) \cong \nu^q(G)/\Delta^q(G)$. If  $G$ is $q$-perfect, then we have
 $\Delta^q(G)=1$ and so $\tau^q(G) \cong \nu^q(G)$. Consequently, the given presentation is a presentation of
$\nu^q(G)$.
	\end{proof}
	\end{thm}

	\subsection{A polycyclic presentation for the $q$-tensor square of a polycyclic group}

By all we have seen, a method for determining a consistent polycyclic presentation for the $q$-tensor square
$G \otimes^q G$ from a given consistent polycyclic presentation of $G$ consists of:
\begin{alg}
\begin{itemize}
\item[$a)$] Determine a consistent polycyclic presentation for $G \wedge^q G$.
\item[$b)$] Determine a consistent polycyclic presentation for $\tau^q(G)$.
\item[$c)$] Determine a consistent polycyclic presentation for $\nu^q(G)$.
\item[$d)$] Determine a consistent polycyclic presentation for the subgroup $\Upsilon^q(G)$ of $\nu^q(G)$.
\end{itemize}
\end{alg}

Step $(a)$ is a direct application of the method for computing a central extension in Sec.~\ref{gast}.
If $G=F_n/R$ a consistent polycyclic presentation of $G$, then we can determine a consistent polycyclic
presentation for $E_q(G)=F_n/[F_n, R^q] R^q$ and we get
$G \wedge^q G$ as the subgroup $(E_q(G))^\prime (E_q(G))^q$.
Step  $(b)$ is thus a direct application of the method developed in Sec.~\ref{sec:4}.
Step $(c)$ is obtained by another application of the method for computing a central extension in Sec.~\ref{gast}, in
order
to compute $\mathfrak{E}^q(\tau^q(G))$ which, by Theorem~\ref{teo16}, is isomorphic to $\nu^q(G)$.
Finally, step $(d)$ is an application of standad methods for compting presentations of subgroups of polycycliclly
presented groups.


\renewcommand{\refname}{REFERENCES}

\end{document}